\documentclass[10pt, a4paper]{amsart}

\usepackage{enumerate, mathtools}
\usepackage{hyperref}
\hypersetup{
    colorlinks=true,
    linkcolor=blue,
    filecolor=blue,      
    urlcolor=blue,
    citecolor=blue,
}
\usepackage{amsmath}
\usepackage{amsthm}
\usepackage{amssymb}
\usepackage{stmaryrd}
\usepackage[matha,mathx]{mathabx} 
\usepackage{xcolor}
\usepackage{tikz}
\usepackage{tikz-cd}
\usepackage[labelformat=simple]{subcaption}
\usepackage{url}
\usepackage{bm}

\numberwithin{equation}{section}
\numberwithin{figure}{section}

\newtheorem{thm}{Theorem}[section]
\newtheorem*{thm*}{Theorem}
\newtheorem*{con*}{Conjecture}
\newtheorem{lem}[thm]{Lemma}
\newtheorem{prop}[thm]{Proposition}
\newtheorem{cor}[thm]{Corollary}

\newtheorem{conj}[thm]{Conjecture}

\theoremstyle{definition}
\newtheorem{defn}[thm]{Definition}
\newtheorem{remark}[thm]{Remark}

\newtheorem{quest}[thm]{Question}

\newcommand{\N}{\mathbb{N}}

\newcommand{\F}{\mathbb{F}}
\newcommand{\Z}{\mathbb{Z}}

\renewcommand{\phi}{\varphi}

\newcommand{\pbinom}[3][]{
    \genfrac{[}{]}{0pt}{}{#2}{#3}_{\ifthenelse{\isempty{#1}}{p}{#1}}
}
\renewcommand{\leq}{\leqslant}
\renewcommand{\geq}{\geqslant}
\renewcommand{\epsilon}{\varepsilon}

\renewcommand{\d}{\mathrm{d}}

\allowdisplaybreaks

\newcommand{\LoF}{\mathcal{L}}
\newcommand{\propL}{\overline{\LoF}}

\newcommand{\covecs}{\mathcal{C}}
\newcommand{\topes}{\mathcal{T}}
\newcommand{\FL}{\mathcal{F}}
\newcommand{\propF}{\overline{\FL}}
\newcommand{\A}{\mathcal{A}}

\newcommand{\cfHP}{\mathsf{cfHP}}
\newcommand{\EP}[2]{E_{#1}^{\mathsf{#2}}}
\newcommand{\zero}{\mathrm{z}}

\title{On the geometry of flag Hilbert--Poincar\'e series for matroids}

\author{Lukas K\"uhne}
\address{Max Planck Institute for Mathematics in the Sciences, Leipzig, Germany}
\address{Fakult\"at f\"ur Mathematik, Universit\"at Bielefeld, D-33501 Bielefeld, Germany}
\email{lkuehne@math.uni-bielefeld.de}

\author{Joshua Maglione} 
\address{Fakult\"at f\"ur Mathematik,
  Universit\"at Bielefeld, D-33501 Bielefeld, Germany}
\email{jmaglione@math.uni-bielefeld.de}

\thanks{This research was partially supported by DFG-grant VO~1248/4-1 project number~373111162.}

\keywords{Coarse flag polynomial, Eulerian polynomials, Igusa zeta functions, oriented matroids}

\makeatletter
\@namedef{subjclassname@2020}{\textup{2020} Mathematics Subject Classification}
\makeatother

\subjclass[2020]{05B35, 52C40}

\begin{document}

\date{\today} 
\maketitle

\vspace{-2em}

\begin{abstract}
    We extend the definition of coarse flag Hilbert--Poincar\'e series to
    matroids; these series arise in the context of local Igusa zeta functions
    associated to hyperplane arrangements. We study these series in the case of
    oriented matroids by applying geometric and combinatorial tools related to
    their topes. In this case, we prove that the numerators of these series are
    coefficient-wise bounded below by the Eulerian polynomial and equality holds
    if and only if all topes are simplicial. Moreover this yields a sufficient
    criterion for non-orientability of matroids of arbitrary rank. 
\end{abstract}

\section{Introduction}

The flag Hilbert--Poincar\'e series associated to a hyperplane arrangement,
defined in~\cite{MV:fHPseries}, is a rational function in several variables
connected to local Igusa zeta functions~\cite{BSY:hyperplane-zf}. In fact,
polynomial substitutions of the variables of the flag Hilbert--Poincar\'e series
also yield motivic zeta functions associated to matroids~\cite{JKU:motivic};
see~\cite{vanderVeer} for the topological analog. There are also substitutions
yielding so-called ask zeta functions associated to certain modules of
matrices~\cite{RV:CICO}; see~\cite[Prop.~4.8]{MV:fHPseries}. The
analytic and arithmetic properties of these zeta functions are, therefore,
heavily influenced by the combinatorics of the flag Hilbert--Poincar\'e series.
Here, we bring in combinatorial tools to better understand features of this
series.

We consider a specialization in variables $Y$ and $T$, called the coarse flag
Hilbert--Poincar\'e series, which seems to have remarkable combinatorial
properties. In~\cite{MV:fHPseries}, it was shown that for most Coxeter
hyperplane arrangements, the numerator of this specialization at $Y=1$ is equal
to an Eulerian polynomial. We generalize this to the setting of oriented
matroids, a combinatorial abstraction of the face structure determined by real
hyperplane arrangements. We show that the numerator can be better understood
from the geometry of the topes, which are analogs of the chambers for real
hyperplane arrangements. This settles a question by Voll and the second
author~\cite[Quest.~1.7]{MV:fHPseries} for the case of real arrangements, asking
about which properties of a hyperplane arrangement guarantee the equality to
Eulerian polynomials mentioned above.

\subsection{Flag Hilbert--Poincar\'e series for matroids}

Let $M$ be a matroid, with ground set $E$, and $\LoF(M)$ its lattice of flats,
with bottom and top elements denoted by $\hat{0}$ and $\hat{1}$, respectively.
Relevant definitions concerning matroids and oriented matroids are given in
Section~\ref{sec:preliminaries}. Let $\mu_M : \LoF(M) \rightarrow \Z$ be the
M\"obius function on~$\LoF(M)$, where $\mu_M(\hat{0})=1$ and
$\mu_M(X)=-\sum_{X'<X}\mu_M(X')$. A well-studied invariant of a matroid $M$ is
the \emph{Poincar\'e polynomial}  
\begin{align*}
    \pi_M(Y) &= \sum_{X\in \LoF(M)} \mu_M(X)(-Y)^{r(X)} ,
\end{align*}
where $r(X)$ is the \emph{rank} of $X$ in $\mathcal{L}(M)$, viz.~one less than
the maximum over the number of elements of all flags from $\hat{0}$ to $X$. If
$M$ is realized by a hyperplane arrangement $\mathcal{A}$, then its Poincar\'e
polynomial captures topological and algebraic properties of
$\mathcal{A}$~\cite{OT:hyperplanes}.

For a poset $P$ let $\Delta(P)$ be the set of flags of $P$, and let
$\Delta_k(P)\subseteq \Delta(P)$ be the set of flags of size $k$.
If $P$ has a bottom element $\hat{0}$ and a top element $\hat{1}$ set
$\overline{P} = P\setminus\{\hat{0},\hat{1}\}$.
The \emph{flag Poincar\'e polynomial} associated to $F =
(X_1 < \cdots < X_{\ell}) \in \Delta(\propL(M))$, with $\ell\geq 0$, is the
product of Poincar\'e polynomials on the minors determined by $F$, 
\begin{align*} 
    \pi_F(Y) &= \prod_{k=0}^{\ell} \pi_{M/X_{k}|X_{k+1}}(Y),
\end{align*}
where $X_0=\hat{0}$ and $X_{\ell+1}=\hat{1}$. Here, $M/X_k$ is the contraction
of $X_k\subseteq E$ from $M$, and $M|X_{k+1}$ is the restriction of $M$ to
$X_{k+1}\subseteq E$. The lattice $\LoF(M/X_{k}|X_{k+1})$ is isomorphic to the
interval $[X_{k}, X_{k+1}]$ in $\LoF(M)$.

\begin{defn}
The \emph{coarse flag Hilbert--Poincar\'e series} of a matroid~$M$ is
\begin{align*} 
\cfHP_M(Y, T) &= \dfrac{1}{1-T}\sum_{F\in\Delta(\propL(M))} \pi_F(Y) \left(\dfrac{T}{1-T}\right)^{|F|}  = \dfrac{\mathcal{N}_M(Y, T)}{(1 - T)^{r(M)}} .
\end{align*}
We call $\mathcal{N}_M(Y, T)$ the \emph{coarse flag polynomial}:
\begin{align*}
	\mathcal{N}_M(Y, T) = \sum_{F\in\Delta(\propL(M))}  \pi_F(Y) T^{|F|}(1-T)^{r(M)-1-|F|}.
\end{align*}
\end{defn}

We call a matroid $M$ \emph{orientable} if there exists an oriented matroid
whose underlying matroid is $M$. An orientable matroid $M$ is \emph{simplicial}
if $M$ has an oriented matroid structure such that the face lattice of every
tope is a Boolean lattice---equivalently, for real hyperplane arrangements every
chamber is a simplicial cone; see details in
Section~\ref{subsec:oriented-matroids}. For example, all Coxeter arrangements
are simplicial.

\subsection{Main results}

For rational polynomials $f(T)=\sum_{k\geq 0} a_kT^k$ and $g(T) = \sum_{k\geq 0}
b_kT^k$, we write $f(T)\leq g(T)$ if $a_k\leq b_k$ for all $k\geq 0$. We write
$f(T)<g(T)$ to mean $f(T)\leq g(T)$ and $f(T)\neq g(T)$. 

The \emph{Eulerian polynomials} $\EP{r+1}{A}(T)$ and $\EP{r+1}{B}(T)$ are equal
to the $h$-polynomials of the barycentric subdivisions of the boundaries of the
$r$-dimensional simplex and the cross-polytope,
respectively~\cite[Thm.~11.3]{Petersen:EulerianNumbers}. The Eulerian
polynomials are also defined by Coxeter-theoretic descent
statistics~\cite[Sec.~11.4]{Petersen:EulerianNumbers}.
In~\cite[Thm.~D]{MV:fHPseries}, it was shown that for all Coxeter arrangements
$\mathcal{A}$ of rank $r$, without an $\mathsf{E}_8$-factor,
$\mathcal{N}_{\mathcal{A}}(1, T)/\pi_{\mathcal{A}}(1)=\EP{r}{A}(T)$. The next
theorem generalizes this result.

\begin{thm}\label{thm:oriented-lower-bound}
    Let $M$ be an orientable matroid of rank $r$. Then 
    \begin{align} \label{eqn:lower-bound-oriented}
        E_r^{\mathsf{A}}(T) \leq \dfrac{\mathcal{N}_M(1, T)}{\pi_M(1)},
    \end{align}
    and equality holds if and only if $M$ is simplicial. Moreover, 
    \begin{align*} 
        \mathcal{N}_M(1, T^{-1}) &= T^{r-1} \mathcal{N}_M(1, T). 
    \end{align*} 
\end{thm}

The key insight in the proof for Theorem~\ref{thm:oriented-lower-bound} is that
in the orientable case $\mathcal{N}_M(1, T)$ is a sum of $h$-polynomials. Each
of the summands is determined by the topes of $M$; see
Proposition~\ref{prop:h-vectors}. Theorem~\ref{thm:oriented-lower-bound} suggests that $\mathcal{N}_M(Y, T)$ is a ``$Y$-twisted'' sum of $h$-polynomials of the topes, and understanding this could address the nonnegativity conjecture of~\cite{MV:fHPseries} in the orientable case.

A byproduct of Theorem~\ref{thm:oriented-lower-bound} is a sufficient condition
for non-orientability of matroids. The rank $3$ case yields an inequality
concerning the the number of rank $2$ flats above every element in $M$.

\begin{cor}\label{cor:non-orientability} 
    Assume $M$ is a simple matroid with rank $3$, and suppose $c$ is the number
    of rank $2$ flats of $M$ and $s$ the sum of their sizes. If
    $3(c-1) < s$, then $M$ is non-orientable. 
\end{cor}

It is known that the Fano matroid is non-orientable, which is also shown by
Corollary~\ref{cor:non-orientability} since it has seven rank $2$ flats, each
containing three elements. There are a number of sufficient conditions for the
non-orientability of matroids. Based on experiments using the database of
non-orientable matroids~\cite{matroiddatabase}, we report that the condition in
Corollary~\ref{cor:non-orientability} is independent from the sufficient
condition in~\cite{CsimaSawyer:6n13} for rank $3$ matroids; see also
\cite[Prop.~6.6.1(i)]{OrientedMatroids}. Moreover,
Corollary~\ref{cor:non-orientability} is related to Corollary 2.6
in~\cite{CuntzGeis} where Cuntz and Geis proved that a rank $3$ arrangement is
simplicial if and only if its underlying matroid satisfies $3(c-1)=s$ in the
notation above.

\subsection{Further questions and conjectures}
The lower bound in~\eqref{eqn:lower-bound-oriented} raises the following
question. How large or how small can the coefficients of the numerator of
$\cfHP_M(1, T)/\pi_M(1)$ be? All of our results and computations suggest the
following.

\begin{conj}\label{conj:bounds}
    For all matroids $M$ of rank $r\geq 3$, 
    \begin{align*} 
        (1 + T)^{r-1} < \dfrac{\mathcal{N}_M(1, T)}{\pi_M(1)} < E_r^{\mathsf{B}}(T). 
    \end{align*}
\end{conj}

We note that $E_1^{\mathsf{A}}(T) = E_1^{\mathsf{B}}(T) = 1$ and
$E_2^{\mathsf{A}}(T) = E_2^{\mathsf{B}}(T)=1 + T$, and all matroids of rank 1 or
2 are both orientable and simplicial. For orientable matroids, the lower bound of
Conjecture~\ref{conj:bounds} holds by Theorem~\ref{thm:oriented-lower-bound}.
Moreover, the upper bound in Conjecture~\ref{conj:bounds} is reminiscent of
similar ``$f$-vector'' bounds proved in~\cite{Fukudaetal,Varchenko}.

\begin{thm}\label{thm:bounds}
    \begin{enumerate}
        \item If Conjecture~\ref{conj:bounds} holds, then the bounds are sharp.
        \item Conjecture~\ref{conj:bounds} holds for all matroids of rank $3$.
        Moreover for all orientable matroids, the upper bound holds for the
        linear term of the polynomials, so Conjecture~\ref{conj:bounds} holds
        for all orientable matroids of rank $4$. 
    \end{enumerate}
\end{thm}

In fact, more is known to hold for $\mathcal{N}_M(Y, T)$ in the case where
$r(M)\leq 3$. We prove, in Proposition~\ref{prop:rank-3}, that the numerator is
nonnegative, palindromic, and when $Y=1$ real-rooted. In particular,
Conjecture~E from~\cite{MV:fHPseries} holds for all central hyperplane arrangements with
rank at most $3$. We are also interested in whether or not these three
properties hold for the numerator of $\mathcal{N}_M(Y, T)$ for all matroids of
rank larger than $3$. For oriented matroids of rank $4$, the polynomial
$\mathcal{N}_M(1, T)$ is real-rooted, which follows from
Theorem~\ref{thm:oriented-lower-bound}. This raises the following general
question.

\begin{quest}\label{quest:realrooted}
    Is the polynomial $\mathcal{N}_M(1, T)$ real-rooted for all matroids $M$?
\end{quest}

Brenti and Welker asked whether the $h$-polynomial of the barycentric
subdivision of a general polytope is real-rooted~\cite{BrentiWelker}. In the
case of real hyperplane arrangements and their associated zonotopes, this
question is related to Question~\ref{quest:realrooted} via our geometric
interpretation of $\mathcal{N}_M(1,T)$ although the precise connection is not
yet well understood.

\subsection{Other matroid invariants}

Given the large number of polynomial matroid invariants, we consider
$\mathcal{N}_M(Y, T)$ in this larger context. The invariant $\mathcal{N}_M(Y,
T)$ is a valuative matroid invariant~\cite[Sec.~14.3]{FS22}, which means that it
behaves well with respect to subdivisions of the matroid base
polytope~\cite{DF10}. To see this, observe that 
\begin{align*} 
    \mathcal{N}_M(Y, T) &= \pi_{M}(Y)(1 - T)^{r(M)-1} + \sum_{X\in \propL(M)} \pi_{M|X}(Y) T (1 - T)^{r(X)-1} \mathcal{N}_{M/X}(Y, T) .
\end{align*}
Using an argument similar to those in Section 8 of~\cite{AS22},
$\mathcal{N}_M(Y, T)$ is a convolution of the Poincar\'e polynomial, and
by~\cite[Thm.~C]{AS22}, it is a valuative matroid invariant. So the coarse flag
polynomial is amenable to techniques recently described by Ferroni and
Schr\"oter in their preprint~\cite{FS22}, and it is a specialization of the
universal $\mathcal{G}$-invariant as proved in~\cite[Thm.~1.4]{DF10}.

Because $\mathcal{N}_M(Y,T)$ is a convolution of the Poincar\'e polynomial or,
similarly, the characteristic polynomial, we briefly consider other invariants
that are also similarly convoluted---such a list appears in Table 1
of~\cite{AS22}. As $\mathcal{N}_M(Y,T)$ and the motivic zeta function
from~\cite{JKU:motivic} are two bivariate specializations of the flag
Hilbert--Poincar\'e series, the two are certainly related but are distinct. The
polynomial $\mathcal{N}_M(Y, T)$ is not a specialization of the Tutte polynomial
of $M$ since $\mathcal{N}_M(Y, T)$ does not satisfy a deletion-contraction
relation. The Kazhdan--Lusztig polynomial, defined in~\cite{EPW16}, does not
seem to be a specialization of the coarse flag polynomial, and similarly Eur's
volume polynomial~\cite[Def.~3.1]{Eur20} does not seem to specialize to the coarse
flag polynomial. The precise relationship between these two polynomials and the
coarse flag polynomial is not entirely clear at this stage.

\subsection{Structure of the article}
We give definitions for matroids and oriented matroids in
Section~\ref{sec:preliminaries}. We prove Theorem~\ref{thm:oriented-lower-bound}
in Section~\ref{sec:oriented}, and Theorem~\ref{thm:bounds} is proved in
Section~\ref{sec:extreme}. Section~\ref{sec:rank3} is devoted to general
matroids of rank $3$. There we also describe a pair of real hyperplane
arrangements with the same coarse flag polynomial and different underlying
matroids (Remark~\ref{rem:AB}), answering a question of Voll and the second
author~\cite{MV:fHPseries}. 

\section{Preliminaries}\label{sec:preliminaries}

We let $\N$ and $\N_0$ be the set of positive and nonnegative integers
respectively. For $n\in\N$, set $[n]=\{1,\dots, n\}$ and $[n]_0=[n]\cup \{0\}$. 

\subsection{Matroids} \label{subsec:matroids}

Let $E$ be a finite set, called the \emph{ground set} and $2^E$ its
power set. A \emph{matroid} $M$ is a pair $(E, \LoF)$ with $\LoF\subseteq  2^E$
its set of \emph{flats} satisfying:
\begin{enumerate} 
    \item $E\in \LoF$, that is $E$ is a flat,
    \item if $X,X'\in \LoF$ are flats then $X\cap X'\in \LoF$ is also a flat,
    and 
    \item if $X\in \LoF$ is a flat then each element of $E\setminus X$ is in
    precisely one of the flats that covers $X$, that is the minimal flats
    strictly containing~$X$.
\end{enumerate}
Ordering the flats by inclusion gives the set of flats $\LoF$ the structure of a
poset, called the \emph{lattice of flats} of the matroid $M$.

One of the main motivations of matroids comes from linear algebra. For a finite
set of hyperplanes $\mathcal{H}=\{H_e ~|~ e \in E\}$ in an $\F$-vector space
$V$, the associated intersection poset $\LoF(\mathcal{H})\coloneqq
\{\bigcap_{e\in S}H_e\mid S\subseteq E\}$ is a poset ordered by reverse
inclusion. The pair $(\mathcal{H},\LoF_{\mathcal{H}})$ is a matroid which is
called an \emph{$\F$-linear matroid}. A matroid $M=(E,\LoF)$ is called
\emph{realizable} over a field $\F$ if there exists an $\F$-linear matroid
$(\mathcal{H},\LoF_{\mathcal{H}})$ for some set of hyperplanes
$\mathcal{H}=\{H_e ~|~ e \in E\}$ with $\LoF = \LoF({\mathcal{H}})$ as posets.
For example, the free matroid $U_{n,n} = ([n], 2^{[n]})$ is realized by the
coordinate hyperplanes over an arbitrary field $\F$ since each $S\subseteq [n]$
is in one-to-one correspondence with an intersection of hyperplanes.

Ordering the flats by inclusion turns $\LoF$ into a ranked lattice. Let
$\LoF_k(M)$ be the set of all flats of rank $k$ for any $k\ge 0$. The rank of
$E$ is the \emph{rank} of the matroid which we denote by $r(M)$. Given a matroid
$M$ we denote its lattice of flats by~$\LoF(M)$. If $\LoF_0(M)=\{\varnothing\}$
and $\LoF_1(M)$ contains only singletons, then $M$ is a \emph{simple matroid}.
For each matroid $M$, there is a unique simple matroid $\mathrm{sim}(M)$ such
that~$\LoF(M)\cong\LoF(\mathrm{sim}(M))$. 

We define two operations on matroids: restriction and contraction relative to a
flat. For $X\in \LoF(M)$, the \emph{restriction} of $M$ to $X$ is the matroid
$M|X := (X, \{X' \in\LoF(M) ~|~ X'\subseteq X\})$. The \emph{contraction} of $X$
from $M$ is the matroid $M/X := (E\setminus X, \{X'\setminus X ~|~ X'\in
\LoF(M), X\subseteq X'\})$.

\subsubsection{Uniform matroids and projective geometries}

We recall two families of matroids which will be important in
Section~\ref{sec:extreme}. The first is the family of \emph{uniform matroids}
$U_{r, n}$ for all $n\geq r\geq 1$. The ground set is $[n]$ and the flats of
$U_{r,n}$ different from $[n]$ comprise all of the $k$-element subsets of $[n]$
for $k\in [r-1]_0$. The second family of matroids is the \emph{projective
geometry} $PG(r-1, q)$ for $r\in \N$ and $q$ a prime power. The ground set is
the set of $1$-dimensional subspaces of $\mathbb{F}_q^r$, and the flats are the
subspaces of $\mathbb{F}_q^r$. It is known that uniform matroids are orientable
and projective geometries are non-orientable for $r\geq 3$.

\subsection{Oriented matroids}\label{subsec:oriented-matroids}

Our notation and terminology for oriented matroids closely
follows~\cite{OrientedMatroids}. We define oriented matroids by their set of
covectors. These are ``vectors'' in symbols $+$, $-$, and $0$, abstracting how a
real hyperplane partitions the vector space into three sets. Each covector
describes a cone relative to each hyperplane. For $X\in\{+,-,0\}^E$, let $-X$ be
defined by replacing $+$ with $-$ and vice versa, keeping the $0$ symbol
unchanged. For $X, Y\in\{+,-,0\}^E$, define $X\circ Y$ via  
\begin{align*} 
    (X\circ Y)_e &= \begin{cases} 
        X_e & \text{if } X_e\neq 0, \\
        Y_e & \text{if } X_e = 0.
    \end{cases} 
\end{align*} 
Lastly, the \emph{separation set} of $X$ and $Y$ is $S(X, Y) = \{e\in E ~|~
X_e=-Y_e\neq 0\}$. A subset $\covecs\subseteq \{+,-,0\}^E$ is a set of
\emph{covectors} of an oriented matroid if $\covecs$ satisfies 
\begin{enumerate} 
    \item $\hat{0}_{\covecs}:=(0,\dots, 0)\in \covecs$, 
    \item if $X\in\covecs$, then $-X\in\covecs$, 
    \item if $X, Y\in \covecs$, then $X\circ Y\in\covecs$,
    \item if $X, Y \in \covecs$ and $e\in S(X, Y)$, then there exists $Z\in
    \covecs$ such that $Z_e=0$ and $Z_f=(X\circ Y)_f=(Y\circ X)_f$ for all
    $f\in E\setminus S(X, Y)$. 
\end{enumerate} 
The pair $M=(E, \covecs)$ is an \emph{oriented matroid} with ground set $E$ and
covectors~$\covecs$. 

The \emph{face lattice} relative to $(E, \covecs)$, denoted by $\FL(\covecs)$,
is the set of covectors together with a (unique) top element $\hat{1}_{\covecs}$
partially ordered by the following relation. For $X, Y\in\covecs$, let $X\leq Y$
if $X_e\in\{0, Y_e\}$ for all $e\in E$. The maximal covectors of
$\propF(\covecs)$ are called \emph{topes}, and the set of all topes is denoted
by $\topes(\covecs)\subseteq \covecs$. 

We define the \emph{zero map} $\zero : \covecs \rightarrow 2^E$ sending $X$ to
$\{e\in E ~|~ X_e=0\}$. The image $\zero(\covecs)\subseteq 2^E$ satisfies the
lattice of flats conditions in Section~\ref{subsec:matroids}, and therefore,
$(E,\zero(\covecs))$ is a matroid~\cite[Prop.~4.1.13]{OrientedMatroids}. We
write $\LoF(M)$ for the lattice of flats of the underlying matroid for $M$. A
matroid $M=(E,\LoF)$ is orientable if there exists an oriented matroid
$(E,\covecs)$ with underlying matroid $M$. 

\section{Matroids of rank $3$}\label{sec:rank3}

We explicitly determine the coarse flag Hilbert--Poincar\'e series for
matroids of rank not larger than $3$. Since $\mathcal{N}_M(Y, T)$ depends only
on $\LoF(M)$, it follows that $\mathcal{N}_M(Y,
T)=\mathcal{N}_{\mathrm{sim}(M)}(Y, T)$. First we require the next lemma, which
follows from the definition of the M\"obius function. 

\begin{lem}\label{lem:poincare-poly-rank3}
    Let $M$ be a simple rank $3$ matroid on $E=[n]$. Let $c$ be the number of
    rank $2$ flats of $M$ and $s$ the sum of their sizes. Then 
    \begin{align*} 
        \pi_M(Y) &= 1 + nY + (s - c)Y^2 + (1 + s - n - c)Y^3 \\
        &= \left(1 + (n-1)Y + (1 + s - n - c)Y^2\right)(1+Y).
    \end{align*}
\end{lem}

\begin{prop}\label{prop:rank-3}
    For a simple rank $3$ matroid $M$ with ground set of size $n$, let $c$ be
    the number of rank $2$ flats of $M$ and $s$ the sum of their sizes. Then
    \begin{align*} 
       \mathcal{N}_M(Y, T) &= \pi_M(Y) + \phi_M(Y)T + Y^3\pi_M(Y^{-1})T^2 ,
    \end{align*} 
    where 
    \begin{align*} 
        \phi_M(Y) &= n+c-2 + (2s - n + c)Y + (2s - n + c)Y^2 + (n+c-2)Y^3.
    \end{align*}
\end{prop}

\begin{proof}
    Recall the formula for the uniform matroid $U_{2,m}$ of rank $2$ on $[m]$, \begin{align*} 
        \pi_{U_{2,m}}(Y) &= (1+Y)(1+(m-1)Y). 
    \end{align*} 
    We first determine the contribution from the flags of size $1$. For
    $X\in\LoF_1(M)$, let~$m_X$ be the number of rank $2$ flats containing $X$.
    Since $M$ has rank $3$, it follows that 
    \begin{align*} 
        \sum_{X\in\propL(M)} \pi_{M/X}(Y) \pi_{M|X}(Y) &= (1 + Y)^2 \sum_{X\in\LoF_2(M)}(1 + (|X|-1)Y) \\
        &\qquad + (1 + Y)^2\sum_{X\in\LoF_1(M)} (1 + (m_X-1)Y) \\
        &= (1 + Y)^2(n + c + (2s - n - c)Y) .
    \end{align*}
    For all maximal flags $F$, $\pi_F(Y) = (1 + Y)^{3}$, so 
    \begin{align*} 
        \mathcal{N}_M(Y, T) &= \pi_M(Y) (1-T)^2 + (1 + Y)^2(n + c + (2s - n - c)Y)(T - T^2) \\
        &\qquad + s(1 + Y)^3T^2 .
    \end{align*} 
    Using Lemma~\ref{lem:poincare-poly-rank3}, the coefficient of $T$, as a
    polynomial in $Y$, is equal to $\phi_M(Y)$, and the coefficient of $T^2$ is 
    \begin{align*} 
        1 + s - n - c + (s-c)Y + nY^2 + Y^3 &= Y^3\pi_M(Y^{-1}). \qedhere
    \end{align*}
\end{proof}

\begin{remark}\label{rem:AB}
    With Proposition~\ref{prop:rank-3}, we answer a question of Voll and the
    second author~\cite[Quest.~6.2]{MV:fHPseries}, about whether there exists a
    distinct pair of arrangements with the same coarse flag polynomial. We
    describe a pair $\mathcal{A}$ and $\mathcal{B}$ of real arrangements in
    Figure~\ref{fig:arrangements} which we found in the database
    of~\cite{Barakat} and are given by:
    \begin{align*} 
        \mathcal{A} &: xyz(x+y)(x-y)(x+2y)(x+z)(y+z)(x+y+z)=0, \\
        \mathcal{B} &: xyz(x+y)(x+2y)(x-2y)(x+z)(2y+z)(2x+2y+z)=0. 
    \end{align*}
    They both contain nine hyperplanes with $c=15$ and $s=39$
    using the above notation. The arrangement $\mathcal{A}$ has exactly two
    planes with three lines of intersection, whereas~$\mathcal{B}$
    has exactly one such plane, so they are nonequivalent.
\end{remark}

\begin{figure}[h]
    \centering 
    \begin{subfigure}[b]{0.49\textwidth}
        \centering
		\includegraphics[width=.75\linewidth]{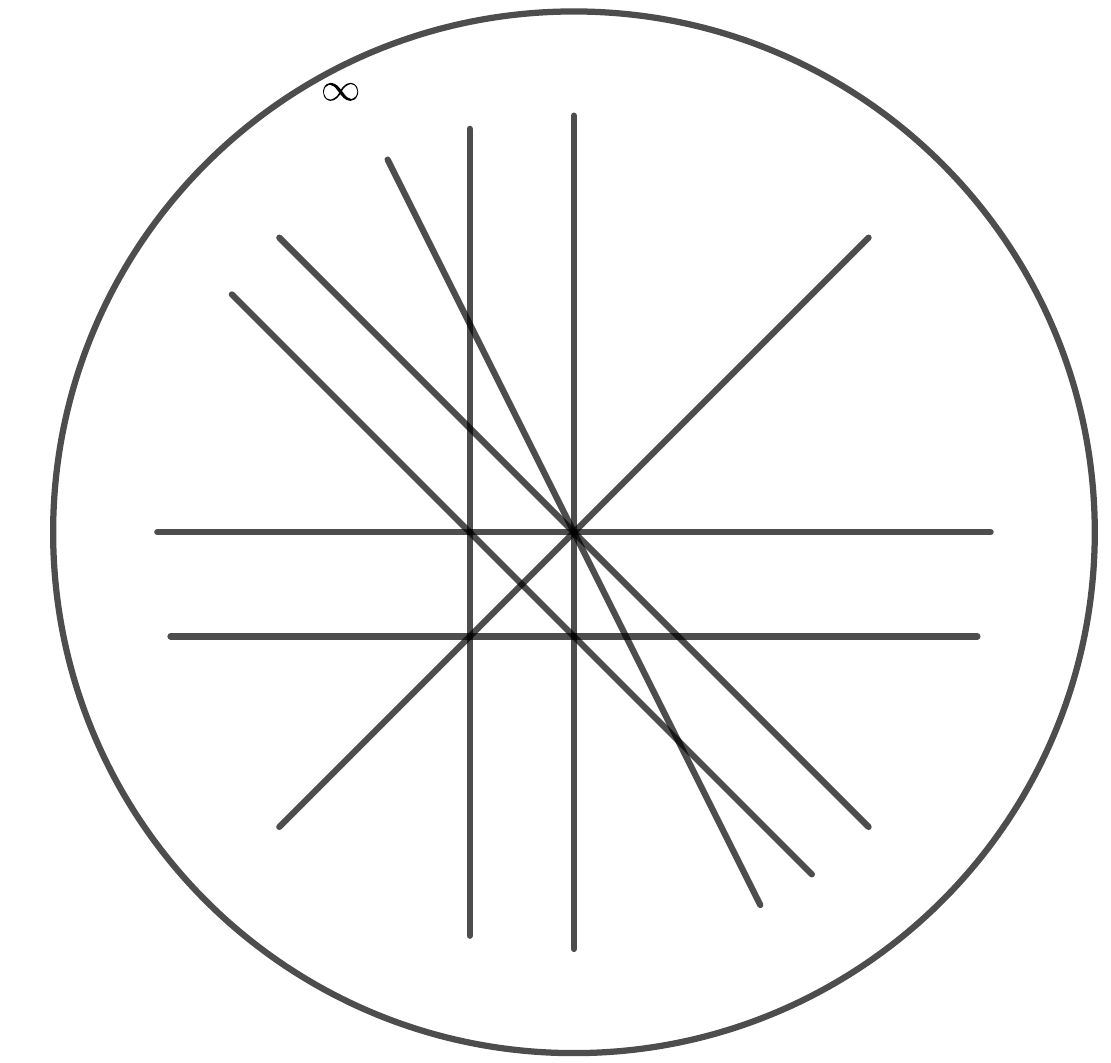}
        \caption{The arrangement $\mathcal{A}$.}
    \end{subfigure}~%
    \begin{subfigure}[b]{0.49\textwidth}
        \centering
        \includegraphics[width=.75\linewidth]{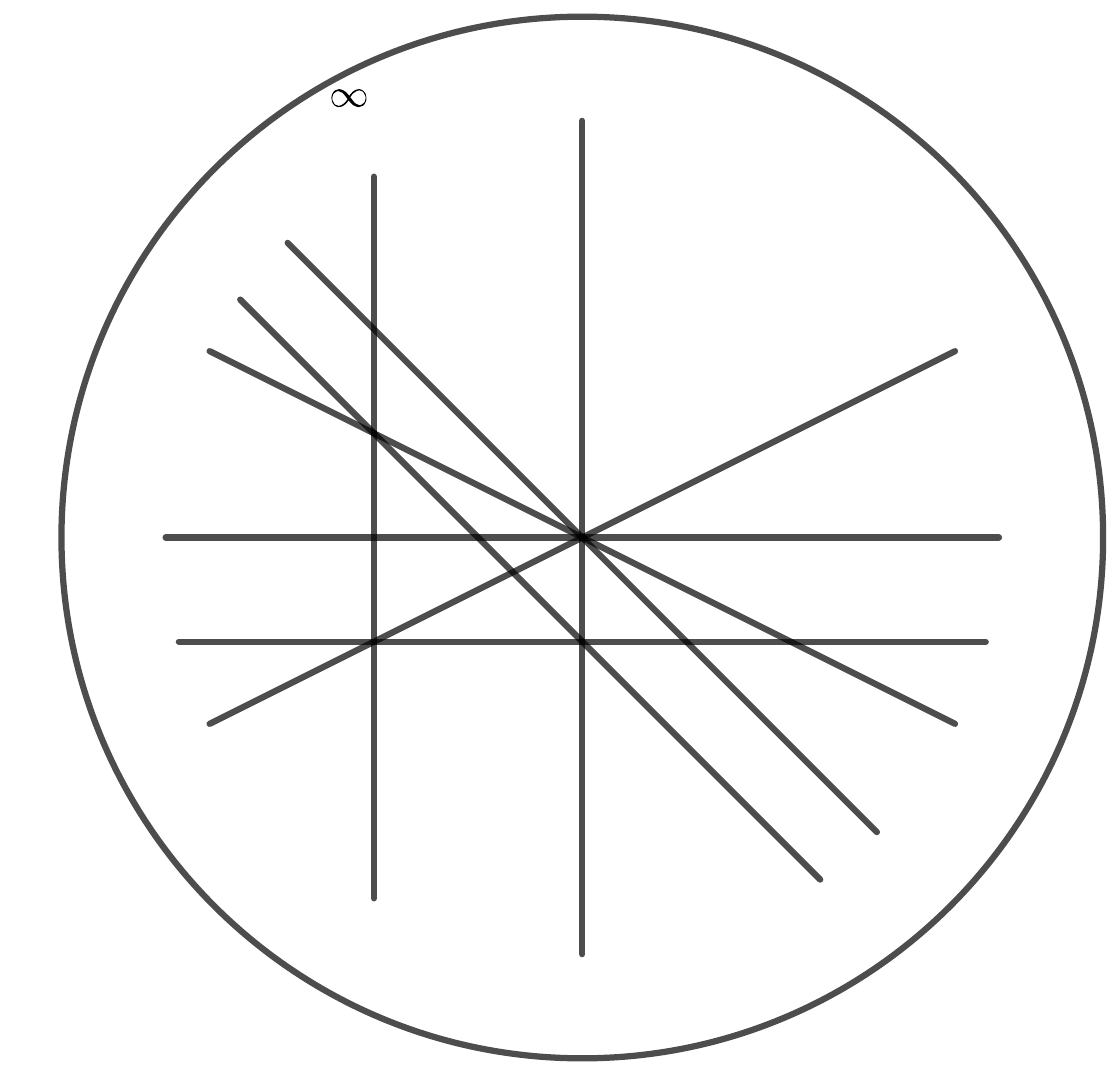}
        \caption{The arrangement $\mathcal{B}$.}
    \end{subfigure}
    \caption{Two projectivized pictures of the arrangements $\A$ and $\mathcal{B}$.}
    \label{fig:arrangements}
\end{figure}

\begin{cor} 
    If $M$ is a simple matroid of rank not larger than $3$, then $\mathcal{N}_M(Y, T)$
    has nonnegative coefficients and satisfies 
    \begin{align*} 
        \mathcal{N}_M(Y^{-1}, T^{-1}) &= Y^{r(M)}T^{r(M)-1}\mathcal{N}_M(Y, T).
    \end{align*}
    Moreover, the polynomial $\mathcal{N}_M(1, T)$ is real-rooted.
\end{cor}

\begin{proof} 
    This is clear if $r(M)=1$ as we assume that $M$ is a simple matroid. If
    $M$ has rank $2$, then $M\cong U_{2, n}$, where $n$ is the size of the
    ground set of $M$. Then, 
    \begin{align*} 
        \mathcal{N}_{U_{2, n}}(Y, T) &= (1+Y)(1+(n-1)Y) + (1+Y)(n-1+Y)T, 
    \end{align*}
    which satisfies the three properties. 
    
    If $r(M)=3$, then from Proposition~\ref{prop:rank-3}, $\mathcal{N}_M(Y, T)$
    satisfies 
    \begin{align*} 
        \mathcal{N}_M(Y^{-1}, T^{-1}) &= Y^{-3}T^{-2} \mathcal{N}_M(Y, T).
    \end{align*}
    The nonnegativity of the coefficients follows if $2s-n+c\geq
    0$. Since every element of the ground set is contained in some rank $2$
    flat, it follows that $2s-n\geq 0$.
    Thus, $\mathcal{N}_M(Y, T)$ has nonnegative coefficients. The discriminant
    of $\mathcal{N}_M(Y, T)$ as a polynomial in $T$ is
    \begin{align*} 
        \left((c+n)^2(1-Y)^2 - 4s(1 - (c+1)Y + Y^2)\right)(1 + Y)^4,
    \end{align*}
    which is positive at $Y=1$.
\end{proof}

\begin{lem}\label{lem:norm-N-rank3}
    For all matroids $M$ with rank $3$, 
    \begin{align*} 
        (1+T)^2 < \dfrac{\mathcal{N}_M(1, T)}{\pi_M(1)} < \EP{3}{B}(T) = 1 + 6T + T^2. 
    \end{align*} 
\end{lem}

\begin{proof}
    Without loss of generality, $M$ is a simple matroid. By
    Proposition~\ref{prop:rank-3}, 
    \begin{align} \label{eqn:rank3-normal}
        \dfrac{\mathcal{N}_M(1, T)}{\pi_M(1)} &= 1 + \left(2 + \dfrac{4(c-1)}{s - (c-1)}\right)T + T^2,
    \end{align} 
    where $c=|\LoF_2(M)|$ and $s=\sum_{X\in\LoF_2(M)}|X|$. Since $s\geq 2c$, 
    \begin{align*} 
        0 &< \dfrac{4(c-1)}{s - (c-1)} < 4. \qedhere
    \end{align*}
\end{proof}

We note that equation~\eqref{eqn:rank3-normal} together with
Theorem~\ref{thm:oriented-lower-bound} proves
Corollary~\ref{cor:non-orientability}.

\section{Oriented matroids}\label{sec:oriented}

Central to the proof of Theorem~\ref{thm:oriented-lower-bound} is the face
lattice of an oriented matroid $M=(E, \covecs)$. Recall from
Section~\ref{subsec:oriented-matroids} that $\covecs$ is the set of covectors,
$\FL(\covecs)$ the face lattice, $\topes(\covecs)$ the set of topes, and
$\zero:\covecs \rightarrow 2^E$ the zero map.

For a poset $P$, let $\widehat{\Delta}(P)$, resp.~$\widehat{\Delta}_k(P)$, be
the set of nonempty flags, resp.~flags of length $k$, ending at a maximal
element of $P$.

A key result that will be applied multiple times for our proof of
Theorem~\ref{thm:oriented-lower-bound} is the Las Vergnas--Zaslavsky Theorem. 
\begin{thm}[{\cite[Theorem 4.6.1]{OrientedMatroids}}] \label{thm:LVZ}
    Let $M=(E, \covecs)$ be an oriented matroid. Then
    \[ 
        |\topes(\covecs)| = \pi_{M}(1) . 
    \]
\end{thm}

\begin{lem}\label{lem:face-poset}
    Let $M=(E,\covecs)$ be an oriented matroid of rank $r$. Then for all
    $k\in [r]$,
    \begin{align*} 
        \left| \widehat{\Delta}_k(\propF(\covecs)) \right| &= \sum_{F\in\Delta_{k-1}(\propL(M))} \pi_F(1) . 
    \end{align*}
\end{lem}

\begin{proof} 
    We prove this by induction on $k$, where the case $k=1$ is
    Theorem~\ref{thm:LVZ}, 
    so we assume it holds for some $k\geq 1$. 
    
    For $X\in\propL(M)$, the matroids $M|X$ and $M/X$ are orientable. The set of
    covectors of $M/X$ is $\covecs/X := \left\{C|_{E\setminus X} : C\in\covecs,
    X\subseteq \zero(C)\right\}$, and the set of covectors of $M|X$ is
    $\covecs|X := \left\{C|_X : C\in\covecs\right\}$. Then by induction 
    \begin{equation}\label{eqn:face-induction}
    \begin{split}
        \sum_{F\in\Delta_k(\propL(M))} \pi_F(1) &= \sum_{X\in \propL(M)}~\sum_{F'\in \Delta_{k-1}(\propL(M|X))} \pi_{M/X}(1)  \pi_{F'}(1) \\
        &= \sum_{X\in \propL(M)} \pi_{M/X}(1) \left| \widehat{\Delta}_k(\propF(\covecs|X)) \right|  \\
        &= \sum_{X\in \propL(M)} \left|\topes(\covecs/X) \right| \left| \widehat{\Delta}_k(\propF(\covecs|X)) \right|.
    \end{split}
    \end{equation}
    The last equation in~\eqref{eqn:face-induction} follows from Theorem~\ref{thm:LVZ}.
    
    Fix $X\in \propF(M)$. The set of topes $\topes(\covecs/X)$ is canonically in
    bijection with the set $\{C \in \covecs : \zero(C) = X\}$. Hence, the set
    $\topes(\covecs/X)\times \widehat{\Delta}_k(\propF(\covecs|X))$ determines a
    flag in $\widehat{\Delta}_{k+1}(\propF(\covecs))$ beginning with a face
    whose zero set is $X$. More precisely, for $C\in\topes(\covecs/X)$ and $F =
    (C_1 < \cdots < C_k) \in \widehat{\Delta}_k(\propF(\covecs|X))$, we define a
    flag $(C' < C_1' < \cdots < C_k')\in
    \widehat{\Delta}_{k+1}(\propF(\covecs))$ such that 
    \begin{align*} 
        C_e' &= \begin{cases} 
            C_e & e\in E\setminus X, \\
            0 & e\in X,
        \end{cases} 
    \end{align*} 
    and for $i\in [k]$, 
    \begin{align*} 
        (C_i')_e &= \begin{cases} 
            C_e & e\in E\setminus X, \\
            (C_i)_e & e\in X.
        \end{cases} 
    \end{align*} 

    Lastly, the number of flags in
    $\widehat{\Delta}_{k+1}(\propF(\covecs))$ beginning with a
    face whose zero set is $X\in\propL(M)$ is $\left|\topes(\covecs/X) \right|
    \left| \widehat{\Delta}_k(\propF(\covecs|X)) \right|$.
    Hence, by~\eqref{eqn:face-induction}, the lemma holds.
\end{proof}

For a finite simplicial complex $\Sigma$, we write $f(\Sigma) := (f_0, \dots,
f_d)\in \N_0^{d+1}$ for the \emph{$f$-vector} of $\Sigma$, where $f_k$ is the
number of $k$-subsets in $\Sigma$---equivalently, the number of
$(k-1)$-dimensional faces. Let $f(\Sigma; T) = \sum_{k=0}^{d} f_k T^k$ be the
$f$-polynomial of $\Sigma$, and let $h(\Sigma; T) := (1 - T)^d f(\Sigma;
T/(1-T))$, which is the \emph{$h$-polynomial} associated to $\Sigma$. The
coefficients of $h(\Sigma; T)$ yield the \emph{$h$-vector} $h(\Sigma)$ of
$\Sigma$.

For a tope $\tau\in\topes(\covecs)$, we define a simplicial complex
$\Sigma(\tau) := \Delta((\hat{0}_{\covecs}, \tau))$, which is the set of flags
in the open interval $(\hat{0}_{\covecs}, \tau)$ in $\mathcal{F}(\covecs)$
ordered by refinement. We write $\Sigma_k(\tau)$ for the flags of $\Sigma(\tau)$
with length $k$. If $M$ is realizable over $\mathbb{R}$, then
$\Sigma(\tau)$ is the barycentric subdivision of the boundary of the chamber
determined by $\tau$. 

\begin{prop}\label{prop:h-vectors}
    Let $M=(E, \covecs)$ be an oriented matroid of rank $r$. Then 
    \begin{align*} 
       \mathcal{N}_M(1, T) &= \sum_{\tau\in\topes(\covecs)} h(\Sigma(\tau); T). 
    \end{align*}
\end{prop}

\begin{proof}
    The flags in $\widehat{\Delta}(\propF(\covecs))$ are
    partitioned into subsets $\widehat{\Delta}((\hat{0}_{\covecs},\tau])$ for
    $\tau\in\topes(\covecs)$, and the latter are in bijection with the flags in
    $\Sigma(\tau)$. Thus, for each $k\in [r-1]_0$, 
    \begin{align*} 
        \left| \widehat{\Delta}_{k+1}(\propF(\covecs)) \right| &= \sum_{\tau\in\topes(\covecs)} \left| \Sigma_k(\tau)\right| . 
    \end{align*}
    Applying Lemma~\ref{lem:face-poset}, we have 
    \begin{align*} 
        \sum_{\tau\in\topes(\covecs)} h(\Sigma(\tau); T) &= \sum_{k=0}^{r-1}\sum_{\tau\in\topes(\covecs)} \left| \Sigma_k(\tau)\right| T^k(1-T)^{r-k-1} \\
        &= \sum_{k=0}^{r-1}~\sum_{F\in \Delta_k(\propL(M))} \pi_F(1) T^k(1-T)^{r-k-1} = \mathcal{N}_M(1, T).\qedhere
    \end{align*}
\end{proof}

In order to prove the lower bound in Theorem~\ref{thm:oriented-lower-bound}, we
work with the cd-index of an (Eulerian) poset. Details can be found
in~\cite[Ch.~3.17]{Stanley:Vol1}.

Let $P$ be a graded poset of rank $n$ with rank function $r : P \rightarrow
[n]_0$. For $S\subseteq [n]_0$, let $P_S=\{x\in P ~|~ r(x)\in S\}$. Set
$\alpha_P(S)$ to be the number of maximal flags in $P_S$, and let
\begin{align*} 
    \beta_P(S) = \sum_{U\subseteq S} (-1)^{|S\setminus U|}\alpha_P(U). 
\end{align*}
Let $a$ and $b$ be two noncommuting variables. For a subset $S\subseteq
\left[n\right]_0$ we define a monomial $u_S$ by setting $u_S=e_0e_1\dots e_n$,
where
\[
e_i=
\begin{cases}
a, &\mbox{if }i\not \in S,\\
b, &\mbox{if }i \in S.
\end{cases}
\]
Using these monomials we can define the \emph{$ab$-index} $	\Psi_P(a,b)$ of the
graded poset $P$ which is the noncommutative polynomial
\[
	\Psi_P(a,b)=\sum_{S\subseteq \left[n\right]_0} \beta_P(S)u_S.
\]
If $P$ is an Eulerian poset, that is every interval in $P$ has an equal number
of elements of even and odd rank, there exists a polynomial $\Phi_P(c,d)$ in the
noncommuting variables $c,d$ such that
\[
	\Psi_P(a,b)=\Phi_P(a+b,ab+ba).
\]
The polynomial $\Phi_P(c,d)$ is called the \emph{$cd$-index} of the Eulerian
poset $P$. For an overview about the $cd$-index see \cite{Bay:cdSurvey}.

If $M=(E,\covecs)$ is an oriented matroid of rank $r$, then for $\tau\in\topes(\covecs)$ and $k\in [r-1]$,
\begin{align}\label{eqn:flag-h-tope}
h_k(\Sigma(\tau)) = \sum_{\substack{S\subseteq [r-1] \\ |S|=k}} \beta_{[\hat{0}_{\covecs}, \tau]}(S).
\end{align}
Therefore, since $[\hat{0}_{\covecs}, \tau]$ is
Eulerian~\cite[Cor.~4.3.8]{OrientedMatroids}, we have 
\begin{equation}\label{eq:cdindex}
	h(\Sigma(\tau);T)  = \Psi_{[\hat{0}_{\covecs}, \tau]}(1,T)=\Phi_{[\hat{0}_{\covecs}, \tau]}(1+T,2T),
\end{equation}
so the $h$-polynomial can be viewed as a coarsening of the cd-index.

\begin{prop}\label{prop:h-lower-bound} 
    Let $M=(E,\covecs)$ be an oriented matroid of rank $r$. Then for all
    $\tau\in\topes(\covecs)$, 
    \begin{align*} 
       \EP{r}{A}(T) &\leq  h(\Sigma(\tau); T),
    \end{align*}
    and equality holds if and only if $\tau$ is a simplicial tope.
\end{prop}

\begin{proof}
	A lattice $P$ of rank $r$ is Gorenstein* if it is Cohen--Macaulay and
	Eulerian. Ehrenborg and Karu~\cite[Cor.~1.3]{EK:cdindex} proved that such a
	lattice $P$ satisfies
	\begin{equation}\label{eqn:cd-index}
	\Phi_{B_r}(c,d)\le \Phi_P(c,d),
    \end{equation}
	where $B_r$ is the Boolean lattice of rank $r$. If $\tau\in\topes(\covecs)$,
	then the interval $[\hat{0}_{\covecs},\tau]$ is both Cohen--Macaulay and
	Eulerian as shown in~\cite[Cor.~4.3.7 \&~4.3.8]{OrientedMatroids}. Thus,
	using \eqref{eq:cdindex} we obtain, by substituting $c=1+T$ and $d=2T$
	in~\eqref{eqn:cd-index}, $\Phi_{B_r}(1+T,2T)  \le h(\Sigma(\tau); T)$. The
	claimed inequality thus follows from $\Phi_{B_r}(1+T,2T)=\EP{r}{A}(T)$ as
	this is the $h$-polynomial of the barycentric subdivision of the
	$r$-dimensional simplex.
	
	If $\tau$ is a simplicial tope, then $[\hat{0}_{\covecs},\tau]$ is a Boolean
	lattice by definition and, thus, $\EP{r}{A}(T) =  h(\Sigma(\tau); T)$.
	On the other hand, suppose we have $\EP{r}{A}(T) =  h(\Sigma(\tau);
	T)$ for the tope $\tau\in\topes(\covecs)$. This implies
	\[
		h_1(\Sigma(\tau)) = 2^r-r-1.
	\]
	By definition we have
	\begin{align*}
		h_1(\Sigma(\tau)) &=\sum_{i=1}^{r-1} \left(\alpha_{[\hat{0}_{\covecs},\tau]} (\{ i \}) -1\right) = \sum_{i=1}^{r-1} \alpha_{[\hat{0}_{\covecs},\tau]} (\{ i \}) -r+1.
	\end{align*}
	 By~\cite[Exercise 4.4 (b)]{OrientedMatroids} we have $\alpha_{[\hat{0}_{\covecs},\tau]} (\{ i \})\ge \binom{r}{i}$.
	Altogether this yields
	\[
		\sum_{i=1}^{r-1} \binom{r}{i} = 2^r-2 = \sum_{i=1}^{r-1} \alpha_{[\hat{0}_{\covecs},\tau]} (\{ i \}) \ge \sum_{i=1}^{r-1}\binom{r}{i}.
	\]
	Thus we obtain $\alpha_{[\hat{0}_{\covecs},\tau]} (\{ r-1 \})=r$ which by~\cite[Exercise 4.4 (c)]{OrientedMatroids} implies that the tope~$\tau$ is simplicial.
\end{proof}

\begin{proof}[Proof of Theorem~\ref{thm:oriented-lower-bound}]
    The proof of the first statement follows from
    Theorem~\ref{thm:LVZ} and Propositions~\ref{prop:h-vectors}
    and~\ref{prop:h-lower-bound}. Since $\beta_{[\hat{0}_{\covecs}, \tau]}(S) =
    \beta_{[\hat{0}_{\covecs}, \tau]}([r-1]\setminus S)$
    by~\cite[Corollary~3.16.6]{Stanley:Vol1}, it follows from
    Equation~\eqref{eqn:flag-h-tope} that $h_k(\Sigma(\tau)) =
    h_{r-k-1}(\Sigma(\tau))$. Hence, the second statement follows by
    Proposition~\ref{prop:h-vectors}. 
\end{proof}

\subsection{Examples}

We compute $\mathcal{N}_M(1, T)$ for some oriented matroids $M$. 

\subsubsection{A uniform matroid with rank $3$}

Consider the matroid $M=U_{3, 4}$. One set of covectors is defined by the real
arrangement given by $xyz(x+y+z)0$. There
are $14=2^4-2$ topes since $(+++-)$ and $(---+)$ are not topes. For instance,
the inequality system given by $x>0,y>0,z>0$ and $x+y+z<0$ is infeasible. The
topes with an even number of $+$ symbols correspond to triangles, and the topes with an
odd number of $+$ symbols correspond to squares. Therefore, there are $8$ triangles and
$6$ squares, so by Proposition~\ref{prop:h-vectors}, 
\begin{align*} 
    \mathcal{N}_M(1, T) &= 8(1 + 4T + T^2) + 6(1 + 6T + T^2) = 14 + 68T + 14T^2.
\end{align*}
By Proposition~\ref{prop:rank-3}, the coarse numerator for $M$ is given by 
\begin{align*} 
    \mathcal{N}_M(Y, T) &= 1 + 4Y + 6Y^2 + 3Y^3 + (8 + 26Y + 26Y^2 + 8Y^3)T \\
    &\qquad + (3 + 6Y + 4Y^2 + Y^3)T^2.
\end{align*}

\subsubsection{A uniform matroid with rank $4$}

Three uniform matroids, which are not isomorphic to the matroids underlying Coxeter
arrangements, in~\cite{MV:fHPseries} had the seemingly rare property that
$\mathcal{N}_M(1, T)/\pi_M(1) \in \Z[T]$. These are the uniform matroids
$U_{r,n}$ for
\[
(r,n)\in \{(4, 5), (4,7), (4,8)\},
\]
and we consider $(r,n)=(4,7)$. From Proposition~\ref{prop:h-vectors}, this integrality
condition is equivalent to the integrality of the average of the $h$-vectors. To
do this computation, we used the hyperplane arrangement
package~\cite{KastnerPanizzut:polymake} of \textsf{polymake} version
4.4~\cite{polymake}.

The matroid $U_{4, 7}$ can be realized as a hyperplane arrangement in
$\mathbb{R}^4$, whose hyperplanes are given by
\[
x_1x_2x_3x_4(x_1+x_2+x_3+x_4)(x_1+2x_2+3x_3+4x_4)(x_1+3x_2+2x_2+5x_4)=0.
\]
There are five different polytopes corresponding to chambers of this arrangement, and they can be seen in Figure~\ref{fig:polys-U47}.
The chambers are $4$-dimensional cones over these polytopes.

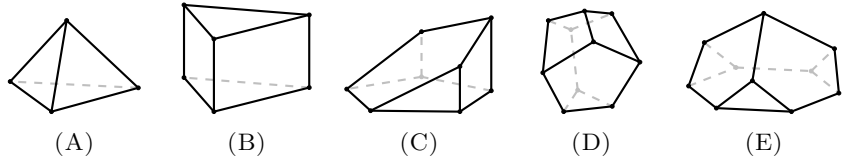
\begin{figure}[h]
    \centering 
    \begin{subfigure}[b]{0.17\textwidth}
        \centering
        \begin{tikzpicture}%
            [x={(-0.979853cm, 0.048658cm)},
            y={(-0.199719cm, -0.238383cm)},
            z={(0.000070cm, 0.969952cm)},
            scale=1.000000,
            back/.style={color=black, dashed, thick},
            edge/.style={color=black, thick},
            facet/.style={fill=white,fill opacity=0.75},
            vertex/.style={inner sep=0.5pt,circle,draw=black,fill=black}]
            %
            %
            
            \coordinate (-0.86000, -0.50000, 0.00000) at (-0.86000, -0.50000, 0.00000);
            \coordinate (0.00000, 0.00000, 1.00000) at (0.00000, 0.00000, 1.00000);
            \coordinate (0.00000, 1.00000, 0.00000) at (0.00000, 1.00000, 0.00000);
            \coordinate (0.86000, -0.50000, 0.00000) at (0.86000, -0.50000, 0.00000);
            \draw[edge,back] (-0.86000, -0.50000, 0.00000) -- (0.86000, -0.50000, 0.00000);
            \fill[facet] (0.00000, 1.00000, 0.00000) -- (-0.86000, -0.50000, 0.00000) -- (0.00000, 0.00000, 1.00000) -- cycle {};
            \fill[facet] (0.86000, -0.50000, 0.00000) -- (0.00000, 0.00000, 1.00000) -- (0.00000, 1.00000, 0.00000) -- cycle {};
            \draw[edge] (-0.86000, -0.50000, 0.00000) -- (0.00000, 0.00000, 1.00000);
            \draw[edge] (-0.86000, -0.50000, 0.00000) -- (0.00000, 1.00000, 0.00000);
            \draw[edge] (0.00000, 0.00000, 1.00000) -- (0.00000, 1.00000, 0.00000);
            \draw[edge] (0.00000, 0.00000, 1.00000) -- (0.86000, -0.50000, 0.00000);
            \draw[edge] (0.00000, 1.00000, 0.00000) -- (0.86000, -0.50000, 0.00000);
            \node[vertex] at (-0.86000, -0.50000, 0.00000)     {};
            \node[vertex] at (0.00000, 0.00000, 1.00000)     {};
            \node[vertex] at (0.00000, 1.00000, 0.00000)     {};
            \node[vertex] at (0.86000, -0.50000, 0.00000)     {};
        \end{tikzpicture}
        \caption{}
        \label{fig:U47-simplex}
    \end{subfigure}
    \begin{subfigure}[b]{0.17\textwidth}
        \centering 
        \begin{tikzpicture}%
            [x={(-0.957953cm, 0.076953cm)},
            y={(-0.286926cm, -0.256983cm)},
            z={(-0.000018cm, 0.963347cm)},
            scale=1.000000,
            back/.style={color=black, dashed, thick},
            edge/.style={color=black, thick},
            facet/.style={fill=white,fill opacity=0.75},
            vertex/.style={inner sep=0.5pt,circle,draw=black,fill=black}]
            %
            %
            
            \coordinate (-0.86000, -0.50000, 0.00000) at (-0.86000, -0.50000, 0.00000);
            \coordinate (-0.86000, -0.50000, 1.00000) at (-0.86000, -0.50000, 1.00000);
            \coordinate (0.00000, 1.00000, 0.00000) at (0.00000, 1.00000, 0.00000);
            \coordinate (0.00000, 1.00000, 1.00000) at (0.00000, 1.00000, 1.00000);
            \coordinate (0.86000, -0.50000, 0.00000) at (0.86000, -0.50000, 0.00000);
            \coordinate (0.86000, -0.50000, 1.00000) at (0.86000, -0.50000, 1.00000);
            \draw[edge,back] (-0.86000, -0.50000, 0.00000) -- (0.86000, -0.50000, 0.00000);
            \fill[facet] (0.86000, -0.50000, 1.00000) -- (0.00000, 1.00000, 1.00000) -- (0.00000, 1.00000, 0.00000) -- (0.86000, -0.50000, 0.00000) -- cycle {};
            \fill[facet] (0.00000, 1.00000, 1.00000) -- (-0.86000, -0.50000, 1.00000) -- (-0.86000, -0.50000, 0.00000) -- (0.00000, 1.00000, 0.00000) -- cycle {};
            \fill[facet] (0.86000, -0.50000, 1.00000) -- (-0.86000, -0.50000, 1.00000) -- (0.00000, 1.00000, 1.00000) -- cycle {};
            \draw[edge] (-0.86000, -0.50000, 0.00000) -- (-0.86000, -0.50000, 1.00000);
            \draw[edge] (-0.86000, -0.50000, 0.00000) -- (0.00000, 1.00000, 0.00000);
            \draw[edge] (-0.86000, -0.50000, 1.00000) -- (0.00000, 1.00000, 1.00000);
            \draw[edge] (-0.86000, -0.50000, 1.00000) -- (0.86000, -0.50000, 1.00000);
            \draw[edge] (0.00000, 1.00000, 0.00000) -- (0.00000, 1.00000, 1.00000);
            \draw[edge] (0.00000, 1.00000, 0.00000) -- (0.86000, -0.50000, 0.00000);
            \draw[edge] (0.00000, 1.00000, 1.00000) -- (0.86000, -0.50000, 1.00000);
            \draw[edge] (0.86000, -0.50000, 0.00000) -- (0.86000, -0.50000, 1.00000);
            \node[vertex] at (-0.86000, -0.50000, 0.00000)     {};
            \node[vertex] at (-0.86000, -0.50000, 1.00000)     {};
            \node[vertex] at (0.00000, 1.00000, 0.00000)     {};
            \node[vertex] at (0.00000, 1.00000, 1.00000)     {};
            \node[vertex] at (0.86000, -0.50000, 0.00000)     {};
            \node[vertex] at (0.86000, -0.50000, 1.00000)     {};
        \end{tikzpicture}
        \caption{}
        \label{fig:U47-prism}
    \end{subfigure}
    \begin{subfigure}[b]{0.17\textwidth} 
        \centering 
        \begin{tikzpicture}%
            [x={(0.964008cm, 0.066984cm)},
            y={(-0.265872cm, 0.243004cm)},
            z={(0.000036cm, 0.967710cm)},
            scale=1.000000,
            back/.style={color=black, dashed, thick},
            edge/.style={color=black, thick},
            facet/.style={fill=white,fill opacity=0.75},
            vertex/.style={inner sep=0.5pt,circle,draw=black,fill=black}]
            %
            %
            
            \coordinate (1.00000, 0.00000, 0.00000) at (1.00000, 0.00000, 0.00000);
            \coordinate (0.31000, 0.95000, 0.00000) at (0.31000, 0.95000, 0.00000);
            \coordinate (-0.81000, 0.59000, 0.00000) at (-0.81000, 0.59000, 0.00000);
            \coordinate (-0.81000, -0.59000, 0.00000) at (-0.81000, -0.59000, 0.00000);
            \coordinate (0.31000, -0.95000, 0.00000) at (0.31000, -0.95000, 0.00000);
            \coordinate (1.00000, 0.00000, 1.00000) at (1.00000, 0.00000, 1.00000);
            \coordinate (0.31000, 0.95000, 0.61878) at (0.31000, 0.95000, 0.61878);
            \coordinate (0.31000, -0.95000, 0.61878) at (0.31000, -0.95000, 0.61878);
            \draw[edge,back] (1.00000, 0.00000, 0.00000) -- (0.31000, 0.95000, 0.00000);
            \draw[edge,back] (0.31000, 0.95000, 0.00000) -- (-0.81000, 0.59000, 0.00000);
            \draw[edge,back] (0.31000, 0.95000, 0.00000) -- (0.31000, 0.95000, 0.61878);
            \node[vertex] at (0.31000, 0.95000, 0.00000)     {};
            \fill[facet] (0.31000, -0.95000, 0.61878) -- (-0.81000, -0.59000, 0.00000) -- (0.31000, -0.95000, 0.00000) -- cycle {};
            \fill[facet] (1.00000, 0.00000, 1.00000) -- (0.31000, -0.95000, 0.61878) -- (-0.81000, -0.59000, 0.00000) -- (-0.81000, 0.59000, 0.00000) -- (0.31000, 0.95000, 0.61878) -- cycle {};
            \fill[facet] (0.31000, -0.95000, 0.61878) -- (0.31000, -0.95000, 0.00000) -- (1.00000, 0.00000, 0.00000) -- (1.00000, 0.00000, 1.00000) -- cycle {};
            \draw[edge] (1.00000, 0.00000, 0.00000) -- (0.31000, -0.95000, 0.00000);
            \draw[edge] (1.00000, 0.00000, 0.00000) -- (1.00000, 0.00000, 1.00000);
            \draw[edge] (-0.81000, 0.59000, 0.00000) -- (-0.81000, -0.59000, 0.00000);
            \draw[edge] (-0.81000, 0.59000, 0.00000) -- (0.31000, 0.95000, 0.61878);
            \draw[edge] (-0.81000, -0.59000, 0.00000) -- (0.31000, -0.95000, 0.00000);
            \draw[edge] (-0.81000, -0.59000, 0.00000) -- (0.31000, -0.95000, 0.61878);
            \draw[edge] (0.31000, -0.95000, 0.00000) -- (0.31000, -0.95000, 0.61878);
            \draw[edge] (1.00000, 0.00000, 1.00000) -- (0.31000, 0.95000, 0.61878);
            \draw[edge] (1.00000, 0.00000, 1.00000) -- (0.31000, -0.95000, 0.61878);
            \node[vertex] at (1.00000, 0.00000, 0.00000)     {};
            \node[vertex] at (-0.81000, 0.59000, 0.00000)     {};
            \node[vertex] at (-0.81000, -0.59000, 0.00000)     {};
            \node[vertex] at (0.31000, -0.95000, 0.00000)     {};
            \node[vertex] at (1.00000, 0.00000, 1.00000)     {};
            \node[vertex] at (0.31000, 0.95000, 0.61878)     {};
            \node[vertex] at (0.31000, -0.95000, 0.61878)     {};
        \end{tikzpicture}
        \caption{}
        \label{fig:U47-pinched}
    \end{subfigure} 
    \begin{subfigure}[b]{0.17\textwidth} 
        \centering 
        \begin{tikzpicture}%
            [x={(-0.723200cm, -0.312625cm)},
            y={(0.690639cm, -0.327354cm)},
            z={(-0.000008cm, 0.891687cm)}, 
            scale=1.15,
            back/.style={color=black, dashed, thick},
            edge/.style={color=black, thick},
            facet/.style={fill=white,fill opacity=0.75},
            vertex/.style={inner sep=0.5pt,circle,draw=black,fill=black}]
            %
            %
            
            \coordinate (0.33000, 0.00000, 0.00000) at (0.33000, 0.00000, 0.00000);
            \coordinate (-0.16500, -0.28579, 0.00000) at (-0.16500, -0.28579, 0.00000);
            \coordinate (-0.16500, 0.28579, 0.00000) at (-0.16500, 0.28579, 0.00000);
            \coordinate (0.66000, 0.00000, 0.61878) at (0.66000, 0.00000, 0.61878);
            \coordinate (-0.33000, -0.57158, 0.61878) at (-0.33000, -0.57158, 0.61878);
            \coordinate (-0.33000, 0.57158, 0.61878) at (-0.33000, 0.57158, 0.61878);
            \coordinate (-0.43165, 0.00000, 1.00000) at (-0.43165, 0.00000, 1.00000);
            \coordinate (0.21583, 0.37382, 1.00000) at (0.21583, 0.37382, 1.00000);
            \coordinate (0.21583, -0.37382, 1.00000) at (0.21583, -0.37382, 1.00000);
            \coordinate (0.00000, 0.00000, 1.18524) at (0.00000, 0.00000, 1.18524);
            \draw[edge,back] (0.33000, 0.00000, 0.00000) -- (-0.16500, -0.28579, 0.00000);
            \draw[edge,back] (-0.16500, -0.28579, 0.00000) -- (-0.16500, 0.28579, 0.00000);
            \draw[edge,back] (-0.16500, -0.28579, 0.00000) -- (-0.33000, -0.57158, 0.61878);
            \draw[edge,back] (-0.33000, -0.57158, 0.61878) -- (-0.43165, 0.00000, 1.00000);
            \draw[edge,back] (-0.33000, -0.57158, 0.61878) -- (0.21583, -0.37382, 1.00000);
            \node[vertex] at (-0.16500, -0.28579, 0.00000)     {};
            \node[vertex] at (-0.33000, -0.57158, 0.61878)     {};
            \fill[facet] (0.00000, 0.00000, 1.18524) -- (-0.43165, 0.00000, 1.00000) -- (-0.33000, 0.57158, 0.61878) -- (0.21583, 0.37382, 1.00000) -- cycle {};
            \fill[facet] (0.00000, 0.00000, 1.18524) -- (0.21583, 0.37382, 1.00000) -- (0.66000, 0.00000, 0.61878) -- (0.21583, -0.37382, 1.00000) -- cycle {};
            \fill[facet] (0.21583, 0.37382, 1.00000) -- (0.66000, 0.00000, 0.61878) -- (0.33000, 0.00000, 0.00000) -- (-0.16500, 0.28579, 0.00000) -- (-0.33000, 0.57158, 0.61878) -- cycle {};
            \draw[edge] (0.33000, 0.00000, 0.00000) -- (-0.16500, 0.28579, 0.00000);
            \draw[edge] (0.33000, 0.00000, 0.00000) -- (0.66000, 0.00000, 0.61878);
            \draw[edge] (-0.16500, 0.28579, 0.00000) -- (-0.33000, 0.57158, 0.61878);
            \draw[edge] (0.66000, 0.00000, 0.61878) -- (0.21583, 0.37382, 1.00000);
            \draw[edge] (0.66000, 0.00000, 0.61878) -- (0.21583, -0.37382, 1.00000);
            \draw[edge] (-0.33000, 0.57158, 0.61878) -- (-0.43165, 0.00000, 1.00000);
            \draw[edge] (-0.33000, 0.57158, 0.61878) -- (0.21583, 0.37382, 1.00000);
            \draw[edge] (-0.43165, 0.00000, 1.00000) -- (0.00000, 0.00000, 1.18524);
            \draw[edge] (0.21583, 0.37382, 1.00000) -- (0.00000, 0.00000, 1.18524);
            \draw[edge] (0.21583, -0.37382, 1.00000) -- (0.00000, 0.00000, 1.18524);
            \node[vertex] at (0.33000, 0.00000, 0.00000)     {};
            \node[vertex] at (-0.16500, 0.28579, 0.00000)     {};
            \node[vertex] at (0.66000, 0.00000, 0.61878)     {};
            \node[vertex] at (-0.33000, 0.57158, 0.61878)     {};
            \node[vertex] at (-0.43165, 0.00000, 1.00000)     {};
            \node[vertex] at (0.21583, 0.37382, 1.00000)     {};
            \node[vertex] at (0.21583, -0.37382, 1.00000)     {};
            \node[vertex] at (0.00000, 0.00000, 1.18524)     {};
        \end{tikzpicture}
        \caption{}
        \label{fig:U47-beaut}
    \end{subfigure}
    \begin{subfigure}[b]{0.17\textwidth} 
        \centering 
        \begin{tikzpicture}%
            [x={(-0.623765cm, -0.245664cm)},
            y={(0.781612cm, -0.196069cm)},
            z={(0.000014cm, 0.949319cm)},
            scale=1.000000,
            back/.style={color=black, dashed, thick},
            edge/.style={color=black, thick},
            facet/.style={fill=white,fill opacity=0.75},
            vertex/.style={inner sep=0.5pt,circle,draw=black,fill=black}]
            %
            %
            
            \coordinate (1.00000, 0.00000, 0.00000) at (1.00000, 0.00000, 0.00000);
            \coordinate (0.50000, -0.86603, 0.00000) at (0.50000, -0.86603, 0.00000);
            \coordinate (-0.50000, -0.86603, 0.00000) at (-0.50000, -0.86603, 0.00000);
            \coordinate (-1.00000, 0.00000, 0.00000) at (-1.00000, 0.00000, 0.00000);
            \coordinate (-0.50000, 0.86603, 0.00000) at (-0.50000, 0.86603, 0.00000);
            \coordinate (0.50000, 0.86603, 0.00000) at (0.50000, 0.86603, 0.00000);
            \coordinate (0.86603, 0.50000, 0.45000) at (0.86603, 0.50000, 0.45000);
            \coordinate (0.00000, -1.00000, 0.45000) at (0.00000, -1.00000, 0.45000);
            \coordinate (-0.86603, 0.50000, 0.45000) at (-0.86603, 0.50000, 0.45000);
            \coordinate (0.00000, 0.00000, 1.06471) at (0.00000, 0.00000, 1.06471);
            \draw[edge,back] (0.50000, -0.86603, 0.00000) -- (-0.50000, -0.86603, 0.00000);
            \draw[edge,back] (-0.50000, -0.86603, 0.00000) -- (-1.00000, 0.00000, 0.00000);
            \draw[edge,back] (-0.50000, -0.86603, 0.00000) -- (0.00000, -1.00000, 0.45000);
            \draw[edge,back] (-1.00000, 0.00000, 0.00000) -- (-0.50000, 0.86603, 0.00000);
            \draw[edge,back] (-1.00000, 0.00000, 0.00000) -- (-0.86603, 0.50000, 0.45000);
            \node[vertex] at (-0.50000, -0.86603, 0.00000)     {};
            \node[vertex] at (-1.00000, 0.00000, 0.00000)     {};
            \fill[facet] (0.00000, 0.00000, 1.06471) -- (0.86603, 0.50000, 0.45000) -- (0.50000, 0.86603, 0.00000) -- (-0.50000, 0.86603, 0.00000) -- (-0.86603, 0.50000, 0.45000) -- cycle {};
            \fill[facet] (0.86603, 0.50000, 0.45000) -- (1.00000, 0.00000, 0.00000) -- (0.50000, 0.86603, 0.00000) -- cycle {};
            \fill[facet] (0.00000, 0.00000, 1.06471) -- (0.86603, 0.50000, 0.45000) -- (1.00000, 0.00000, 0.00000) -- (0.50000, -0.86603, 0.00000) -- (0.00000, -1.00000, 0.45000) -- cycle {};
            \draw[edge] (1.00000, 0.00000, 0.00000) -- (0.50000, -0.86603, 0.00000);
            \draw[edge] (1.00000, 0.00000, 0.00000) -- (0.50000, 0.86603, 0.00000);
            \draw[edge] (1.00000, 0.00000, 0.00000) -- (0.86603, 0.50000, 0.45000);
            \draw[edge] (0.50000, -0.86603, 0.00000) -- (0.00000, -1.00000, 0.45000);
            \draw[edge] (-0.50000, 0.86603, 0.00000) -- (0.50000, 0.86603, 0.00000);
            \draw[edge] (-0.50000, 0.86603, 0.00000) -- (-0.86603, 0.50000, 0.45000);
            \draw[edge] (0.50000, 0.86603, 0.00000) -- (0.86603, 0.50000, 0.45000);
            \draw[edge] (0.86603, 0.50000, 0.45000) -- (0.00000, 0.00000, 1.06471);
            \draw[edge] (0.00000, -1.00000, 0.45000) -- (0.00000, 0.00000, 1.06471);
            \draw[edge] (-0.86603, 0.50000, 0.45000) -- (0.00000, 0.00000, 1.06471);
            \node[vertex] at (1.00000, 0.00000, 0.00000)     {};
            \node[vertex] at (0.50000, -0.86603, 0.00000)     {};
            \node[vertex] at (-0.50000, 0.86603, 0.00000)     {};
            \node[vertex] at (0.50000, 0.86603, 0.00000)     {};
            \node[vertex] at (0.86603, 0.50000, 0.45000)     {};
            \node[vertex] at (0.00000, -1.00000, 0.45000)     {};
            \node[vertex] at (-0.86603, 0.50000, 0.45000)     {};
            \node[vertex] at (0.00000, 0.00000, 1.06471)     {};
        \end{tikzpicture}
        \caption{}
        \label{fig:U47-truncated}
    \end{subfigure}
    \caption{The five different polytopes arising as chambers in the $U_{4, 7}$ arrangement.}
    \label{fig:polys-U47}
\end{figure}

There are a total of 84 chambers; 22 are simplices, 22 are triangular prisms, 30
are the polytopes seen in Figure~\ref{fig:U47-pinched}, six are the polytopes
seen in Figure~\ref{fig:U47-beaut}, and four are truncated simplices as seen in
Figure~\ref{fig:U47-truncated}. The $h$-vectors of the barycentric subdivisions
are palindromic, and the first values different from $1$ are $11$, $17$, $23$,
$29$, and $29$ respectively. Thus, 
\begin{align*} 
    \mathcal{N}_{U_{4, 7}}(1, T) &= 22(1 + 11T + 11T^2 + T^3) + 22(1 + 17T + 17T^2 + T^3) \\
    &\qquad + 30(1 + 23T + 23T^2 + T^3) + (4+6)(1 + 29T + 29T^2 + T^3). 
\end{align*}
This has the nice coincidence that 
\begin{align*} 
    \mathcal{N}_{U_{4, 7}}(1, T) &= 84 (1 + 19T + 19T^2 + T^3).
\end{align*} 
Curiously, $(1,19,19,1)$ is the $h$-vector of the barycentric subdivision of the pyramid over a pentagon.

\subsection{Rank $4$ oriented matroids}

In this section, we prove that $\mathcal{N}_M(1, T)$, for an oriented matroid
$M=(E, \covecs)$ of rank $4$, is bounded above coefficient-wise by
$\pi_M(1)\EP{r}{B}(T)$. The next lemma determines the coefficients of
$\mathcal{N}_M(1, T)$ in terms of the face lattice of $M$. To simplify notation,
we define 
\begin{align*} 
    \mathfrak{f}_{k}(\covecs) := \left|\widehat{\Delta}_{k+1}\left(\propF(\covecs)\right) \right| = \left|\left\{F\in \Delta_{k+1}\left(\propF(\covecs)\right) : F \text{ ends at a tope} \right\}\right|. 
\end{align*}
For $f(T) = \sum_{k\geq 0}a_kT^k$, let $f(T)[T^k]=a_k$.

\begin{lem}\label{lem:N-terms}
    Let $M=(E,\covecs)$ be an oriented matroid of rank $r$. For $\ell\in
    [r-1]_0$, 
    \begin{align*} 
        \mathcal{N}_M(1, T)[T^{\ell}] &= \sum_{k=0}^\ell (-1)^{\ell-k} \mathfrak{f}_{k}(\covecs) \binom{r-k-1}{\ell-k}.
    \end{align*}
\end{lem}

\begin{proof} 
    Let $\tau\in\topes(\covecs)$ and let $f(\Sigma(\tau); T)$ be the
    $f$-polynomial. The coefficient of $T^{\ell}$ in the $h$-polynomial of $\Sigma(\tau)$ is 
    \begin{align*} 
        h_\ell(\Sigma(\tau)) = \sum_{k=0}^\ell (-1)^{\ell-k}f_{k}(\Sigma(\tau)) \binom{r-k-1}{\ell-k}. 
    \end{align*}
    By Proposition~\ref{prop:h-vectors}, 
    \begin{align*} 
        \mathcal{N}_M(1, T)[T^{\ell}] &= \sum_{\tau\in\topes(\covecs)} \sum_{k=0}^\ell (-1)^{\ell-k}f_{k}(\Sigma(\tau)) \binom{r-k-1}{\ell-k} \\
        &= \sum_{k=0}^\ell (-1)^{\ell-k} \mathfrak{f}_{k}(\covecs) \binom{r-k-1}{\ell-k} . \qedhere 
    \end{align*}
\end{proof}

\begin{prop}\label{prop:rank-4}
    If $M$ is an orientable matroid of rank $r\geq 3$, then 
    \begin{align*} 
        \mathcal{N}_M(1, T)[T] < \pi_M(1) \EP{r}{B}(T)[T] . 
    \end{align*}
    If $M$ is of rank $4$, then $\mathcal{N}_M(1, T)< \pi_M(1)
    E_4^{\mathsf{B}}(T)$. 
\end{prop}

\begin{proof}
    From Theorem~\ref{thm:oriented-lower-bound}, $\mathcal{N}_M(1, T)$ has
    degree $r-1$ and is palindromic. Therefore it suffices to just prove
    the inequality between the linear coefficients. 
    
    Suppose $\covecs$ is a set of covectors such that $M=(E, \covecs)$ is an
    oriented matroid. The number $\mathfrak{f}_1(\covecs)$ counts the flags of
    length two in $\propF(\covecs)$ which end at a tope, and
    $\mathfrak{f}_0(\covecs)=|\topes(\covecs)|$. Using Proposition~4.6.9
    of~\cite{OrientedMatroids}, we have the following inequality:
    \begin{align}\label{eqn:Varchenko-bound}
        \mathfrak{f}_1(\covecs) &< \sum_{j=0}^{r-2} 2^{r-1-j}\binom{r-1}{j} |\topes(\covecs)|. 
    \end{align}
    Using~\cite[Sec.~13.1]{Petersen:EulerianNumbers}, one can express the terms
    of $\EP{r}{B}(T)$ in terms of alternating sums. The linear term is,
    thus, $3^{r-1}-r$.
    \begin{align*} 
        \mathcal{N}_M(1, T)[T] &= \mathfrak{f}_1(\covecs) - (r-1)\mathfrak{f}_0(\covecs) &  (\text{Lemma~\ref{lem:N-terms}}) & \\
        &< \left(\sum_{j=0}^{r-2} 2^{r-1-j}\binom{r-1}{j} |\topes(\covecs)|\right) - (r-1)|\topes(\covecs)| &  (\text{Eq.~\ref{eqn:Varchenko-bound}}) & \\ 
        &= (3^{r-1}-1)|\topes(\covecs)| - (r-1)|\topes(\covecs)| \\
        &= \pi_M(1) \EP{r}{B}[T] .  & & 
    \end{align*}
    The penultimate equality is seen by counting, in two different ways, the
    number of ways to color $r-1$ balls with three colors. 
\end{proof}

\section{Extremal families}\label{sec:extreme}

In this section, we prove Theorem~\ref{thm:bounds} in two parts by constructing
infinite families of matroids whose normalized coarse flag polynomial is
arbitrarily close to the bounds given in Conjecture~\ref{conj:bounds}. The upper
bound is witnessed by the family of uniform matroids, and the lower
bound is witnessed by the family of finite projective geometries. See
Section~\ref{subsec:matroids} for definitions.

In what follows, we compute limits of univariate polynomials of a fixed degree.
Identifying degree $d$ polynomials $a_0 + a_1T + \cdots + a_dT^d$ with the
points $(a_0,\dots, a_d)\in\mathbb{R}^{d+1}$, the limit is determined using the
Euclidean norm in $\mathbb{R}^{d+1}$.

\subsection{The upper bound}

The next lemma relates the type $\mathsf{B}$ Eulerian polynomial $\EP{r}{B}$
with the type $\mathsf{A}$ Eulerian polynomial
$\EP{r}{A}$, which may be of independent interest. 

\begin{lem}\label{lem:EulerianAB}
    For $r\in \N$, 
    \begin{align}\label{eqn:EulerianAB}
        \EP{r}{B}(T) &= (1 - T)^{r-1} + \sum_{k = 1}^{r - 1} 2^k \binom{r - 1}{k} T(1 - T)^{r-k-1}E_k^{\mathsf{A}}(T).
    \end{align}
\end{lem}

\begin{proof}
    Let $P_r(T)$ denote the polynomial on the right side
    in~\eqref{eqn:EulerianAB}. It is clear that $P_1(T)=E^{\mathsf{B}}_1(T)=1$,
    so we assume $r\geq 2$. Recall two recurrence relations concerning the two
    Eulerian polynomials~\cite[Thms.~1.4 \&~13.2]{Petersen:EulerianNumbers};
    namely, 
    \begin{align*} 
        E_{r+1}^{\mathsf{A}} &= (1 + rT)\EP{r}{A} + T(1 - T)\dfrac{\d \EP{r}{A}}{\d T},  \\
        E_{r+1}^{\mathsf{B}} &= (1 + (2r-1)T) \EP{r}{B} + 2T(1 - T)\dfrac{\d \EP{r}{B}}{\d T}. 
    \end{align*}
    We will prove that $P_r(T)$ satisfies the type $\mathsf{B}$ recurrence
    relation, and thus the lemma will follow. 

    Applying the type $\mathsf{A}$ recurrence relation on Eulerian polynomials,
    we have 
    \begin{align*}
        2T(1-T)\dfrac{\d P_r}{\d T} &= 2 T (1-r)(1 - T)^{r-1} \\
        &\qquad + \sum_{k=1}^{r-1} 2^{k+1} \binom{r-1}{k} T (1 - T)^{r-k-1} \left(E_{k+1}^{\mathsf{A}} - rT E_k^{\mathsf{A}}\right).
    \end{align*}
    Lastly, we have  
    \begin{align*} 
        & (1 + (2r-1)T)P_r(T) + 2T(1 - T)\dfrac{\d P_r}{\d T}(T) \\
        &= (1 - T)^r + 2T(1 - T)^{r-1} + \sum_{k=1}^{r-1} 2^{k+1}\binom{r-1}{k}T(1-T)^{r-k-1}E_{k+1}^{\mathsf{A}} \\
        &\qquad + \sum_{k=1}^{r-1} 2^{k}\binom{r-1}{k}T(1-T)^{r-k}E_{k}^{\mathsf{A}} \\
        &= (1 - T)^r + 2T(1 - T)^{r-1} + 2^rT\EP{r}{A} + 2(r-1)T(1-T)^{r-1} \\
        &\qquad + \sum_{k=2}^{r-1} 2^{k}\binom{r-1}{k-1}T(1-T)^{r-k}E_{k}^{\mathsf{A}} + \sum_{k=2}^{r-1} 2^{k}\binom{r-1}{k}T(1-T)^{r-k}E_{k}^{\mathsf{A}} \\
        &= (1 - T)^r + \sum_{k=1}^{r} 2^{k}\binom{r}{k}T(1-T)^{r-k}E_{k}^{\mathsf{A}} = P_{r+1}(T). \qedhere
    \end{align*}
\end{proof}

\begin{prop}\label{prop:uniform-limit}
    For $r\in\N$,
    \begin{align*} 
        \lim_{m\rightarrow\infty} \dfrac{\mathcal{N}_{U_{r,m}}(1, T)}{\pi_{U_{r,m}}(1)} &= E^{\mathsf{B}}_r(T). 
    \end{align*}
\end{prop}
\begin{proof}
    From~\cite[Proposition~6.9]{MV:fHPseries}, 
    \begin{align*} 
        \dfrac{\mathcal{N}_{U_{r,m}}(1, T)}{\pi_{U_{r,m}}(1)} &= (1 - T)^{r-1} + \dfrac{\sum_{\ell = 1}^{r-1} \sum_{k = 0}^{r-\ell-1}\binom{m}{\ell}\binom{m-\ell-1}{k}2^{\ell} T (1- T)^{r-\ell-1}E_{\ell}^{\mathsf{A}}}{\sum_{k=0}^{r-1} \binom{m-1}{k}} \\ 
        &= (1 - T)^{r-1} + \sum_{\ell = 1}^{r-1} 2^{\ell} T (1- T)^{r-\ell-1}E_{\ell}^{\mathsf{A}}\left(\dfrac{\binom{m}{\ell}\sum_{k = 0}^{r-\ell-1}\binom{m-\ell-1}{k}}{\sum_{k=0}^{r-1} \binom{m-1}{k}}\right) .
    \end{align*} 
    For $\ell\in [r-1]$, it follows that 
    \begin{align*} 
        \lim_{m\rightarrow\infty} \dfrac{\binom{m}{\ell}\sum_{k = 0}^{r-\ell-1}\binom{m-\ell-1}{k}}{\sum_{k=0}^{r-1} \binom{m-1}{k}} &= \binom{r-1}{\ell}. 
    \end{align*}
    Therefore, 
    \begin{align}\label{eqn:limit}
        \lim_{m\rightarrow \infty} \dfrac{\mathcal{N}_{U_{r,m}}(1, T)}{\pi_{U_{r,m}}(1)} &= (1 - T)^{r-1} + \sum_{\ell = 1}^{r-1}2^{\ell}\binom{r-1}{\ell}  T (1- T)^{r-\ell-1}E_{\ell}^{\mathsf{A}}. 
    \end{align} 
    By Lemma~\ref{lem:EulerianAB}, the right side of~\eqref{eqn:limit} is
    $E^{\mathsf{B}}_r(T)$.
\end{proof}

\subsection{The lower bound}

For an indeterminate $X$, a nonnegative integer $r$, and $0\leq k\leq r$, we set
\begin{align*} 
    \binom{r}{k}_X = \dfrac{(1-X^r)(1-X^{r-1})\cdots (1-X^{r-k+1})}{(1 - X)(1-X^2) \cdots (1 - X^k)} \in \Z[X].
\end{align*}
For $I\subseteq [r-1]$, where $I=\{i_1 < \cdots < i_k\}$, we set $i_0=0$,
$i_{k+1}=r$, and
\begin{align*} 
    \binom{r}{I}_X = \prod_{m=1}^{|I|} \binom{i_{m+1}}{i_m}_X = \binom{r}{i_k}_X \binom{i_k}{i_{k-1}}_X \cdots \binom{i_2}{i_1}_X.
\end{align*}
The number of $k$-dimensional subspaces of $\mathbb{F}_q^r$ is equal to
$\binom{r}{k}_q$. For $I\subseteq [r-1]$, the number of flags $0=V_0\subset V_1
\subset \cdots \subset V_k \subset V_{k+1} = \mathbb{F}_q^r$ with $I=\{\dim(V_j)
~|~ j\in [k]\}$ is equal to $\binom{r}{I}_q$. Throughout this subsection, we
will assume that when $I\subseteq [r-1]$, then $I = \{i_1,\dots, i_k\}$, where
$i_j < i_{j+1}$ for $j\in [k-1]$. We define the polynomial 
\begin{align*} 
    \gamma_I(X, Y) &= \prod_{\ell=0}^{|I|}~\prod_{m=0}^{i_{\ell+1} - i_\ell - 1} (1 + X^mY).
\end{align*}

Let $M=PG(r-1, q)$ be the matroid determined by the projective geometry of
dimension $r$ and order $q$. Thus, the lattice of flats of $M$ is isomorphic to
the subspace lattice of the finite vector space $\mathbb{F}_q^r$. The Poincar\'e
polynomial of $M$ is $\pi_M(Y)=\gamma_{\varnothing}(q, Y)$ since the M\"obius
function values of a flat of rank $k$ in $\LoF(M)$ is $(-1)^kq^{\binom{k}{2}}$.

\begin{lem} \label{lem:cfHP-PG}
    For $r\in\N$ and a prime power $q$, 
    \begin{align*} 
        \mathcal{N}_{PG(r-1, q)}(Y, T) = \sum_{I\subseteq [r-1]} \binom{n}{I}_q  \gamma_{I}(q, Y) T^{|I|}(1 - T)^{r-1-|I|}.
    \end{align*}
\end{lem}

\begin{proof} 
    Let $M=PG(r-1, q)$. It follows that $\Delta(\propL(M))$ is in bijection with
    the set of flags in $\mathbb{F}_q^r$ with proper nontrivial subspaces. Let
    $I\subseteq [r-1]$, and suppose $F$ and $F'$ are flags in $\mathbb{F}_q^r$
    of length $|I|$ containing proper nontrivial subspaces, whose dimensions are
    given by $I$. It follows that $\pi_F(Y) = \pi_{F'}(Y)$ since the intervals
    in $\Delta(\propL(M))$ determined by $F$ and $F'$ are isomorphic. In
    particular, $\pi_F(Y)$ is determined by $I$, the set of (proper nontrivial)
    subspace dimensions. Furthermore, each interval is isomorphic to the lattice
    of flats of some projective geometry of order $q$ with smaller dimension.
    Therefore, $\pi_F(Y) = \gamma_I(q, Y)$. 

    Since the number of flags $F$ with a given set of subspace dimensions
    $I\subseteq [r-1]$ is $\binom{r}{I}_q$, it follows that 
    \begin{align*} 
        \mathcal{N}_{M}(Y, T) &= \sum_{F\in\Delta(\propL(M))} \pi_F(Y) T^{|F|}(1 - T)^{r-1-|F|} \\
        &= \sum_{I\subseteq [r-1]} \binom{r}{I}_q \gamma_{I}(q, Y) T^{|I|}(1 - T)^{r-1-|I|} . \qedhere
    \end{align*} 
\end{proof}

\begin{prop}\label{prop:PG-limit}
    Let $r\in \N$ and 
    \begin{align*} 
        \lim_{q\rightarrow \infty} \dfrac{\mathcal{N}_{PG(r-1, q)}(1, T)}{\pi_{PG(r-1, q)}(1)} = (1 + T)^{r-1},
    \end{align*}
    where $q$ is assumed to be a prime power.
\end{prop}

\begin{proof}
    Let $X$ be an indeterminate, and note that for $I\subseteq [r-1]$, 
    \begin{align*} 
        \lim_{X\rightarrow \infty} \dfrac{\binom{n}{I}_X \gamma_I(X, 1)}{\gamma_{\varnothing}(X, 1)} = 2^{|I|}.
    \end{align*}
    Set $M=PG(r-1,q)$, and since $\pi_M(Y) = \phi_{\varnothing}(q, Y)$, by
    Lemma~\ref{lem:cfHP-PG}, 
    \begin{align*}
        \lim_{q\rightarrow \infty} \dfrac{\mathcal{N}_{M}(1, T)}{\pi_{M}(1)} &= \sum_{I\subseteq [r-1]} 2^{|I|} T^{|I|}(1 - T)^{r-1-|I|} \\
        &= \sum_{k=0}^{r-1} \binom{r-1}{k} (1-T)^{r-1-k} (2T)^k \\
        &= (1+T)^{r-1}. \qedhere 
    \end{align*}
\end{proof}

\subsection{Proof of Theorem~\ref{thm:bounds}}

The first part follows from Propositions~\ref{prop:uniform-limit}
and~\ref{prop:PG-limit}, and the second part follows from
Propositions~\ref{prop:rank-3} and~\ref{prop:rank-4}. \hfill $\square$

\section*{Acknowledgements}

We thank Luis Ferroni, Raman Sanyal, Benjamin Schr\"oter and Christopher Voll
for inspiring discussions. We are also grateful to the anonymous referees for
their helpful feedback. 

\bibliography{Simplicial}
\bibliographystyle{abbrv}

\end{document}